\documentclass[12pt]{article}
\usepackage{amssymb,amsmath}
\usepackage{xcolor}
\usepackage{amsfonts}
\usepackage{dsfont}
\usepackage{amsmath,amssymb,amsthm}
\usepackage{eucal}
\usepackage{latexsym}
\usepackage{amsbsy}
\usepackage{amsmath}
\usepackage{amscd}
\usepackage{amsgen}
\usepackage{amsopn}
\usepackage{amstext}
\usepackage{ifthen}
\usepackage{amsxtra}
\usepackage{amsfonts}
\usepackage{float}
\usepackage{hyperref}
\def \R{{\hbox{\vrule width 0.6pt height 6.8pt depth -.2pt\kern-0.2pt
R}}}
\def \P{{\hbox{\vrule width 0.6pt height 6.8pt depth -.2pt\kern-0.2pt
P}}}

\usepackage[english]{babel}
\usepackage[latin1]{inputenc}
\usepackage{latexsym}

\def \R {\mathbb R}

\def \P {\mathbb P}

\def \t {t_0(x_0)}

\def\MOD#1{{|\kern -.16em |\kern -.16em | #1 | \kern -.16em |\kern
 -.16em |}}
\def \epsilon {\varepsilon}

\newtheorem{theo}{\bf THEOREM}[section]
\newtheorem{lem}[theo]{\bf LEMMA}
\newtheorem{pro}[theo]{\bf PROPOSITION}

\newtheorem{rem}[theo]{\bf REMARK}
\newcommand{\eps}{\varepsilon}

\numberwithin{equation}{section}

\title{ \textbf{\Large Single point blow-up and final profile  for a
    perturbed nonlinear heat equation with a gradient  and a non-local
    term} }
\author{\large Bouthaina Abdelhedi\\
   \textit{\small Department of mathematics, Faculty of Sciences of Sfax, Tunisia}\\
   \large Hatem Zaag\\
  \textit {\small Universit\'e Paris 13, Sorbonne Paris Cit\'e},\\
   \textit{\small LAGA, CNRS (UMR 7539), F-93430, Villetaneuse, France}
 }

\begin{document}
\maketitle






\begin{center}
\textit{Dedicated to the memory of Ezzeddine Zahrouni.}
\end{center}

\begin{abstract} We consider in this paper a perturbation of the standard semilinear heat equation by a term involving the space derivative and a non-local term.  In some earlier work \cite{AZ}, we constructed a blow-up solution for that equation, and showed that it blows up (at least) at the origin. We also derived the so called  ``intermediate blow-up profile''. In this paper,  we prove the single point  blow-up property and  determine the final blow-up profile.
\end{abstract}
~~
\section{Introduction}
We consider the problem
\begin{equation}\label{eq_u}
\left \{
\begin{array}{lcl}
 u_{t}&=&\Delta u+|u|^{p-1}u+\mu |\nabla u|\displaystyle \int_{B(0, |x|)}|u|^{q-1},\\
 u(0)&=&u_0\in W^{1, \infty}(\R^N),
\end{array}
\right.
\end{equation}
where $u=u(x,t)\in \R$, $x\in \R^N$ and the parameters $p, q$ and $\mu$ are such that \begin{equation}\label{hyp}
\displaystyle  p>3, \quad \frac{N}{2}(p-1)+1<q<\frac{N}{2}(p-1) +\frac{p+1}{2},\quad \mu \in \R.
\end{equation}
 When $\mu=0$, we recover the standard semilinear heat equation with power nonlinearity,
\begin{equation}\label{equ}
u_t = \Delta u +|u|^{p-1}u,
\end{equation}
 which has attracted a lot of attention in the last 50 years (see the book \cite{QS} by Quittner and Souplet), in particular, as a model for the study of blow-up in PDEs. Although the analysis of \eqref{equ} is far from being trivial, one may feel that such an equation is too idealized, and may not capture a lot of features one may encounter in real life (parabolic) models. For that reason, some authors tried to study perturbations of that equation with different kind of terms, aiming to be closer to more realistic models. In particular, we mention the following perturbation by a nonlinear gradient term
\begin{equation}\label{equ-bis}
u_t = \Delta u +|u|^{p-1}u+\mu|\nabla u|^q,
\end{equation}
first introduced by Chipot and Weissler \cite{chipot}. We also mention the perturbations involving non-local (or integral) terms, as we encounter in PDEs modeling Micro Electro-Mechanical Systems (MEMS):

\begin{equation}\label{equ-ter}
u_{t}=\Delta u+\displaystyle \frac{\lambda}{(1-u)^2(1+\gamma\int_\Omega\frac{1}{1-u}dx)^2},
\end{equation}
where $\frac{1}{1-u}$ may blow up in finite time (see Duong and Zaag \cite{DZ} and the references therein).
Specific difficulties arise in the study of blow-up for both equations \eqref{equ-bis} and \eqref{equ-ter}, as one may see from the constructions of singular solutions in Tayachi and Zaag \cite{TZ} (see also the note \cite{TZ1}) and  Duong and Zaag \cite{DZ}.
\bigskip

In this paper, following our earlier work \cite{AZ} on \eqref{eq_u}, we intend to consider more complicated perturbations of equation \eqref{equ}. Specifically,  since it shows a product between a gradient term and an integral term, we felt that equation \eqref{eq_u} can be the next challenge to study blow-up for parabolic PDEs.
\bigskip

Equation $(\ref{eq_u})$ is wellposed in the weighted  functional  space $W^{1, \infty}_\beta(\R^N)$ defined as follows:
\begin{equation}
 W^{1, \infty}_\beta(\R^N)=\{g, \; (1+|y|^\beta) g\in L^\infty, \;(1+|y|^\beta)\nabla  g\in L^\infty\},
 \end{equation}
where
\begin{equation}\label{beta}
 0\le \beta <\frac{2}{p-1},\;\mbox{if}\;\mu=0\; \mbox{ and }\;\displaystyle \frac{N}{q-1}<\beta<\frac{2}{p-1},\;\mbox{if}\; \mu\neq 0,
\end{equation}
as one may see from Appendix $C$ in \cite{AZ}.\\
In \cite{AZ}, we constructed a solution  $u(x,t)$ for equation $(\ref{eq_u})$ which blows up in finite time $T$ at $a=0$, and we proved that the solution behaves as follows: for all $(x,t)\in \R^N \times [0, T)$,
\begin{equation}\label{b_u}
  \displaystyle \!\!\left|u(x,t)\!-\!(T\!-\!t)^{-\frac{1}{p-1}}f\left(\frac{x}{\sqrt{(T\!-\!t)|\log(T\!-\!t)|}}\right)\right|\!\leq\! \frac{C}{1+(\frac{|x|^2}{T-t})^{\frac{\beta}{2}}}\frac{(T-t)^{-\frac{1}{p-1}}}{|\log(T-t)|^\frac{1-\beta}{2}},
   \end{equation}
   and
\begin{equation}\label{b_nabla_u}
\displaystyle\!\!\! \left|\nabla\! u(x,t)\!-\!\frac{(T\!-\!t)^{\!-\!\frac{1}{2}\!-\!\frac{1}{p\!-\!1}}}{\sqrt{|\log(T\!-\!t)|}}\!\nabla \!f\left(\frac{x}{\sqrt{(T\!-\!t)|\log(T\!-\!t)|}}\right)\right|\!\!\leq\!\! \frac{C}{1\!+\!(\frac{|x|^2}{T\!\!-t})^{\frac{\beta}{2}}}\frac{(T\!-\!t)^{-\frac{1}{2}\!-\!\frac{1}{p\!-\!1}}}{|\log(T\!-\!t)|^\frac{1-\beta}{2}}
\end{equation}
(note that $0\le \beta<1$ from \eqref{beta} and \eqref{hyp}),
where the ``intermediate'' profile $f$ is given by
\begin{equation}\label{profil}
f(z)=(p-1+b|z|^2)^{-\frac{1}{p-1}}, \;\mbox{for all} \;z\in \R^N, \; \mbox{with}\;\, b=\frac{(p-1)^2}{4p}.\end{equation}
Our argument in \cite{AZ} is a non trivial adaptation of the pioneering work performed for equation \eqref{equ} by
Bricmont and Kupiainen \cite{bricmont} and Merle and Zaag \cite{MZ97} (see also the note \cite{MZ96}). More precisely, the proof is given in the similarity variables setting: we linearize the PDE around the profile candidate $f$ defined in \eqref{profil} (actually, around a small perturbation of $f$), then, we control the nonpositive part of the spectrum thanks to the decaying properties of the linear operator. As for the positive eigenvalues, we control them thanks to a topological argument based on index theory. Note that already when $\mu=0$, the profile in \eqref{b_u} is sharper than the profile derived in \cite{MZ97}, in the sense that we divide here the bound by $1+(\frac{|x|^2}{T-t})^{\frac{\beta}{2}}$.

\bigskip
Despite this sharp estimate, we left two questions unanswered in \cite{AZ}:
\begin{itemize}
\item Is the origin the only blow-up point of the constructed solution?
\item Can we derive an equivalent of the "final profile" (namely, $u^*(x) \equiv \displaystyle\lim_{t\to T}u(x,t)$) near the origin?
\end{itemize}
\bigskip

In this paper, we positively answer these two open questions. In fact, we adapt the technique developed for equation \eqref{equ} by  Giga and Kohn \cite{giga3}, in order  to obtain the single point blow-up result for equation $(\ref{eq_u})$ and give the description of the blow-up
final profile $u^\ast$ such that $u(x,t)\to u^\ast $ in $C^1$  of every compact set   of $\R^N\setminus \{0\}$. More precisely,  we prove the following Theorem:
\begin{theo}\label{main_thm}
Let $\mu\in \R$ and $p$, $q$ be two real numbers such that \eqref{hyp} holds.\\
Assume that   equation $(\ref{eq_u})$ has a solution $u$ which  blows up  at the origin in some finite time $T$ and satisfies  \eqref{b_u} and \eqref{b_nabla_u} for some $\beta$ satisfying \eqref{beta}.
Then,
\begin{enumerate}
\item It holds that $(u,\nabla u)$ blows up only at the origin.
\item For all $x\neq 0$, $u(x,t)\to u^\ast (x)$ as $t\to T$ in $C^1$  of every compact  of $\R^N\setminus \{0\}$, with
\begin{equation}\label{prof-u}
  u^\ast (x)\sim \left[\frac{8p|\log|x||}{(p-1)^2|x|^2}\right]^{\frac{1}{p-1}} 
  \;\mbox{as}\; x\to 0,
\end{equation}
and
for $|x|$ small, we have
\begin{equation}\label{prof-nabla-u}
\displaystyle |\nabla u^\ast(x)|\leq C |x|^{-\frac{p+1}{p-1}}|\log|x||^{\frac{p+3}{4(p-1)}}.
\end{equation}
\end{enumerate}
\end{theo}
\begin{rem}
\begin{enumerate}
\item As for the intermediate profile (i.e. valid for $0\le t<T$) given in \eqref{b_u}-\eqref{b_nabla_u}, the final profile (i.e. for $t=T$) given in \eqref{prof-u}-\eqref{prof-nabla-u} shows no change in the behavior, with respect to the standard semilinear heat equation \eqref{equ}. However, one should keep in mind that the proof is much more complicated, as one already sees from \cite{AZ} and the present paper.
\item From \eqref{b_u}, we know that $u$ blows up at the
  origin. Unfortunately, we cannot say if $\nabla u$ blows up at the
  origin too. In addition, 
  we could only obtain an upper bound in \eqref{prof-nabla-u}. We conjecture that \eqref{prof-u} holds also after differentiation in space. Unfortunately, our estimates are still not sharp enough to prove that.
\item As we have already pointed-out in \cite{AZ}, the single-point blow-up for $u$ is trivial, if $\mu\le 0$.
Indeed, it happens that we are able to ensure that the constructed solution is nonnegative, hence, our solution is a nonnegative subsolution of
the standard semilinear heat equation \eqref{equ}, and the argument of Giga and Kohn in Section 2 page 850 of \cite{giga3}
allows us to conclude that $u$ blows up only at the origin.
However, this leaves open the 
non blow-up of $\nabla u$ outside the origin
when $\mu\le 0$, and for both $u$ and $\nabla u$, when $\mu>0$, not mentioning the
derivation of the final blow-up profile for general $\mu\in \R$.
\end{enumerate}
\end{rem}
The paper is organized as follows. In Section $2$, we control the non-local term in equation \eqref{eq_u}. In Section $3$, we prove  the single-point blow-up property. Finally, in Section $4$, we describe the final profile.
\section{Control of the non-local term}
This section is devoted to the control of the non-local term
 in equation \eqref{eq_u}.
We consider $u(t)$ a solution of $(\ref{eq_u})$  which blows up  in finite time $T$ at the point $a=0$ and satisfies $(\ref{b_u})$ and $(\ref{b_nabla_u})$.
We consider  $\delta>0$ which will be fixed small enough such that various estimates hold. For $K_0>0$ to be fixed large enough later and any $x_0\neq 0$, we define $\t \in [0,T)$ by
\begin{eqnarray}\label{x_0}
|x_0|=K_0\sqrt{(T-\t) |\log(T-\t)| }, \quad &\mbox{if}&\;  0<|x_0|\leq \delta, \\
\t=t_0(\delta), \hspace*{5.5cm} &\mbox{if}& \; |x_0|> \delta.\nonumber
\end{eqnarray}
Note that $\t\to T$ as $x_0\to 0$ and $\t$ is uniquely defined if $\delta$ is sufficiently small.\\
Let us introduce for each 
$x_0\neq 0$
a rescaled version of $u$
\begin{equation}\label{v}
v(x_0, \xi, \tau)=(T-\t)^{\frac{1}{p-1}}u(x,t),
\end{equation}
where
\begin{equation}\label{scaling}
x\!=\!x_0\!+\!\xi\sqrt{T\!-\!\t}, \, t\!=\!\t+\tau(T\!-\!\t),\,\xi \in \R^N, \, \tau \in [- \frac{\t}{T\!-\!\t}, 1).
\end{equation}
We also introduce
\begin{equation}\label{w}
w=\nabla v.
\end{equation}
Since $u$ is a solution of equation $(\ref{eq_u})$, it follows that $v$ and $w$ satisfy
\begin{equation}
  v_{\tau}  =\Delta v+|v|^{p-1}v+\mu (T-\t)^\gamma|w|\displaystyle \int_{B_0}|v|^{q-1},\label{eq_v}\\
 \end{equation}
 \begin{equation}
 w_{\tau}  =\Delta w+p|v|^{p-1}w+\mu (T-\t)^\gamma \nabla(|w|\displaystyle \int_{B_0}|v|^{q-1}),\label{eq_w}
\end{equation}
where
\begin{equation}\label{gamma}
\gamma=\displaystyle \frac{p-q}{p-1}+\frac{N-1}{2}\end{equation}
and   $B_0$ is the ball of center $-\displaystyle 	K_0\sqrt{|\log(T-\t)|}\frac{x_0}{|x_0|}$ and radius\\ $\displaystyle |\xi+K_0\sqrt{|\log(T-\t)|}\frac{x_0}{|x_0|}|.$
Note from \eqref{hyp} that
\begin{equation}\label{hyp_ga}
\gamma >0.
\end{equation}
Let us give an idea of the method used to prove 
Theorem \ref{main_thm}.
We proceed in two steps:
\begin{itemize}
\item First, arguing as Giga and Kohn did in \cite{giga3} [Theorem  2.1 Page 850], we consider $r>0$  and  $\phi_r$  a smooth function supported in $B_r=\{x\in \R^N;  \;  |x|<r\}$ such that
$\phi_r=1$ on $B_{r/2}$ and $0\leq \phi_r\leq 1$. We then consider cut-off versions of the functions, namely $v\phi_r$ and $w\phi_r$. \\
We use an iteration process to show that $v$ and $w$ are bounded,
hence that $\xi=0$ is not a blow-up point of $v$ and $w$. From the
transformation \eqref{v}-\eqref{scaling}, this means that neither $u$
nor $\nabla u$ blow up at $x=x_0$.
Since $u$ blows up at the origin by \eqref{b_u}, this proves the single point blow-up property
for $(u, \nabla u)$.
\item Second, we prove the existence of a blow-up final profile $u^\ast$ such that\\ $u(x,t)\to u^\ast $ as $t\to T$ in $C^1$
of every compact of $\R^N\backslash \{0\}$. Next, we find an equivalent of $u^\ast$ and an upper bound on $\nabla u^*$ near the blow-up point.
\end{itemize}
Note that our equation  $(\ref{eq_u})$ is in the  class of perturbed semilinear  heat equations. However, compared to the previous works, our perturbation is not trivial since we have both a non-local and a  gradient term. Therefore, the previous methods  couldn't be used successfully, mainly because the non-local term in \eqref{eq_v} and \eqref{eq_w} could not be localized, and carries ``long-distance" information over a ball with a far away center  $-\displaystyle 	K_0\sqrt{|\log(T-\t)|}\frac{x_0}{|x_0|}$ which diverges to infinity as $|x_0|\to 0$. So, some crucial modifications are needed. In particular, we need to control this new  non-local term. We state  our key result.
\begin{pro}\label{new_prop} Let $\mu\neq 0$ and $p$, $q$ be two real numbers such that \eqref{hyp} holds.\\
Assume  equation $(\ref{eq_u})$ has a solution $u$ which  blows up  at the origin in some finite time $T$ and satisfies  \eqref{b_u} and \eqref{b_nabla_u} for some $\beta$ satisfying \eqref{beta}.
Then, for all $\eta>0$, there exists a positive constant $C_\eta$ such that for all $(x,t)\in \R^N\times [0, T)$
\begin{equation*}
\displaystyle \int_{B(0, |x|)}|u(x', t)|^{q-1}dx'\leq C_\eta (T-t)^{\gamma -\frac{1}{2}-\eta}.
\end{equation*}
\end{pro}
\begin{proof}
From inequality $(\ref{b_u})$, we have
\begin{eqnarray}
\displaystyle \Big|u(x',t)\Big|^{q-1}&\leq& C
(T-t)^{-\frac{q-1}{p-1}}\Big|f(\frac{x'}{\sqrt{(T-t)|\log(T-t)|}})\Big|^{q-1}\nonumber\\
&&+ C \frac{(T-t)^{-\frac{q-1}{p-1}}}{|\log(T-t)|^\frac{(q-1)(1-\beta)}{2}}\frac{1}{(1+(\frac{|x'|^2}{T-t})^{\frac{\beta}{2}})^{q-1}}.\label{one}
\end{eqnarray}
By  definition  \eqref{profil} of the profile $f$, we see that
\begin{eqnarray*}
\int_{B(0, |x|)}|f(\frac{x'}{\sqrt{(T-t)|\log(T-t)|}})|^{q-1}\leq C \int_{B(0, |x|)}\frac{1}{(1+|\frac{x'}{\sqrt{(T-t)|\log(T-t)|}}|^2)^{\frac{q-1}{p-1}}}dx',\\
\leq  C\big((T-t)|\log(T-t)|\big)^{\frac{N}{2}} \int_{B(0, \frac{|x|}{\sqrt{(T-t)|\log(T-t)|}})}\frac{dy}{(1+|y|^2)^{\frac{q-1}{p-1}}}.
\end{eqnarray*}
Since $2\frac{q-1}{p-1}>N$ from \eqref{hyp}, we have  \begin{eqnarray}\label{I1}
\int_{B(0, |x|)}|f(\frac{x'}{\sqrt{(T-t)|\log(T-t)|}})|^{q-1}
&\leq & C(T-t)^{\frac{N}{2}}|\log(T-t)|^{\frac{N}{2}},
\end{eqnarray}
on the one hand. On the other hand, from the choice of $\beta $ in \eqref{beta}, we have  $(q-1)\beta>N$, and  this yields
 \begin{eqnarray}\label{I2}
 \int_{B(0, |x|)}\frac{1}{(1+(\frac{|x'|^2}{T-t})^{\frac{\beta}{2}})^{q-1}} dx'
&\leq&  C(T-t)^{\frac{N}{2}}\int_{B(0, \frac{|x|}{\sqrt{T-t}})}\frac{dy}{(1+|y|^\beta)^{q-1}}\nonumber\\
&\le&  C(T-t)^{\frac{N}{2}}.
\end{eqnarray}
From \eqref{one}, \eqref{I1} and \eqref{I2}, we get
 \begin{eqnarray*}
\int_{B(0, |x|)}|u(x',t)|^{q-1}dx'
&\leq& C (T-t)^{\frac{N}{2}-\frac{q-1}{p-1}}|\log(T-t)|^{\frac{N}{2}}.
\end{eqnarray*}
Furthermore, for all $\eta>0$, there exists $C_\eta>0$ such that
 \begin{eqnarray*}
\int_{B(0, |x|)}|u(x',t)|^{q-1}dx'
&\leq& C_\eta (T-t)^{\frac{N}{2}-\frac{q-1}{p-1}-\eta}.
\end{eqnarray*}
Using the definition \eqref{gamma} of $\gamma$, we  conclude the proof of Proposition \ref{new_prop}.
\end{proof}
\begin{rem}
We remark that, for $v$ defined in \eqref{v}, the proposition reads  as follows:
\begin{equation}\label{2.6}
\displaystyle (T-\t)^\gamma \int_{B_0}|v(x_0,\xi', \tau)|^{q-1}d\xi'\leq C_\eta (T-\t)^{\gamma -\eta}(1-\tau)^{\gamma -\frac{1}{2}-\eta}.
\end{equation}
\end{rem}
\section{No blow-up under some threshold}
In this section, we prove the single-point blow-up property for $(u,
\nabla u)$, i.e. part 1 of Theorem \ref{main_thm}. Since $u$ blows up
at the origin from \eqref{b_u},  we see from the definitions
\eqref{v}-\eqref{w} that it is enough to prove that $v(x_0 ,\xi, \tau)$
and $w(x_0, \xi, \tau)$ are bounded for $\xi=0$ and $\tau \in [0,1)$, whenever $x_0\neq 0$. More precisely, this is our main tool.
\begin{pro}[No blow-up under some threshold]
 \label{prop_single_blow_up}
 There exist
$\bar \eps$
and $M>1$ with the following property:\\
Assume that for some $\epsilon_0\le \bar \eps$ and $x_0\in
\R^N\backslash \{0\}$, we have
 for all $|\xi|<1$ and $\tau\in [0, 1)$: 
\begin{equation}\label{est_v}
|v(x_0,\xi, \tau )|+\sqrt{1-\tau}|\nabla v(x_0,\xi, \tau )|\leq \eps_0(1-\tau)^{-\frac{1}{p-1}},
\end{equation}
where  $v(x_0,\xi,\tau)$ is defined in \eqref{v}.
Then, for all
$|\xi|\le \frac 1M$
and $\tau\in [0, 1),$ we have
$$|v(x_0,\xi, \tau )|+|\nabla v(x_0,\xi, \tau )|\leq M \eps_0.$$
\end{pro}
\noindent Let us first use this proposition to prove part 1 of Theorem \ref{main_thm}, then, we will prove it.
\begin{proof}[Proof of part 1 of Theorem \ref{main_thm}, assuming that
  Proposition \ref{prop_single_blow_up} holds]~\\
Note first from \eqref{b_u} that $u$ blows up at the
origin, hence $(u,\nabla u)$ does the same. It remains then to prove
that neither $u$ nor $\nabla u$ blow up outside the origin.\\
Consider $\delta>0$ to be fixed later small enough. Consider then $x_0 \neq 0$ and $v(x_0,\xi,\tau)$ defined in \eqref{x_0}-\eqref{scaling}.
 By definition, it is enough to show that this function satisfies the hypotheses of Proposition \ref{prop_single_blow_up} in order to conclude.\\
From \eqref{x_0},  2 cases arise:\\
\label{ProofTh1}- \textbf{Case 1: $|x_0|\le \delta$}. Take $|\xi|<1$ and $\tau \in [0,1)$. By definition \eqref{x_0}-\eqref{scaling}, we write from\eqref{b_u}
\begin{eqnarray*}
&&|v(x_0, \xi, \tau)|=(T-\t)^{\frac 1{p-1}}|u(x,t)|\\
&\le& (T-t_0(x_0))^{\frac 1{p-1}} (T-t)^{-\frac 1{p-1}}
\left\{f\left(\frac x{\sqrt{(T-t)|\log(T-t)|}}\right)+\frac C{|\log(T-t)|^{\frac{1-\beta}2}} \right\}\\
&\le& C(1-\tau)^{-\frac 1{p-1}}
\left\{ f\left(\frac x{\sqrt{(T-\t)|\log(T-\t)|}}\right)
+\frac C{|\log(T-\t)|^{\frac{1-\beta}2}} \right\}\\
& = &C(1-\tau)^{-\frac 1{p-1}}
\left\{ f\left(K_0-\frac 1{\sqrt{|\log(T-\t)|}}\right)
+\frac C{|\log(T-\t)|^{\frac{1-\beta}2}} \right\}\\
& \le & C(1-\tau)^{-\frac 1{p-1}} \left\{ f\left(\frac {K_0}2\right)
+\frac C{|\log(T-\t)|^{\frac{1-\beta}2}} \right\}\\
&\le & \epsilon_0 (1-\tau)^{-\frac 1{p-1}},
\end{eqnarray*}
provided that $\delta$ is small enough, and $K_0$ is large enough.\\
By a similar calculation, we can prove the estimate on $\nabla v$ starting from \eqref{b_nabla_u}.\\
- \textbf{Case 2: $|x_0|>\delta$}. Here, we have $\t=t_0(\delta)$. Concerning $v$, the calculation of the former case works, except for the estimate of $\frac{x_0}{\sqrt{(T-\t)|\log(T-\t)|}}$, which needs this small adjustment:
$
\Big[
\frac{|x_0|}{\sqrt{(T-\t)|\log(T-\t)|}} \ge \frac{\delta}{\sqrt{(T-t_0(\delta))|\log(T-t_0(\delta))|}}=K_0
\Big]
$.
Once again, the estimate for $\nabla v$  follows similarly, starting from \eqref{2.6}.\\
This finishes the proof of Part 1 of Theorem \ref{main_thm}, assuming that Proposition \ref{prop_single_blow_up} holds.
\end{proof}

It remains then to prove Proposition \ref{prop_single_blow_up}. Throughout this section,
 we write $v(\xi,\tau)$ instead of $ v(x_0,\xi,\tau)$, in
order to simplify notations.

\begin{proof}[Proof of Proposition \ref{prop_single_blow_up}]~\\
To prove Proposition \ref{prop_single_blow_up}, we use a cut-off
technique: Let $r\in (0, 1]$,  and consider $\phi_r$  a smooth function supported in $B(0, r)$ such that $\phi_r=1$ on $B(0, \frac{r}{2})$ and $0\leq \phi_r\leq 1.$ We introduce the cut-off  of  the solution $v$ and its  gradient $w=\nabla v$,  by  $v_r= \phi_r v$  and $w_r=\phi_r w$.\\
 Let us sketch the main steps  of the proof of   Proposition \ref{prop_single_blow_up}.
\begin{itemize}
\item In Step $1$, we use the Duhamel formulation of the equation 
satisfied by the  cut-off of the gradient of $v$. Then, starting from
hypothesis \eqref{est_v} of Proposition \ref{prop_single_blow_up} and using
Proposition \ref{new_prop} (actually its consequence \eqref{2.6}), we
prove by an iteration process the existence of some
 $r_1$ such that for all $\tau \in [0,1)$,
\begin{equation}\label{vr1}
\|v(\tau)\|_{L^\infty(B_{r_1})}\leq \frac{\eps_0}{(1-\tau)^{\frac{1}{p-1}}},
\end{equation}
and
\begin{equation}\label{wr1}
\|w(\tau)\|_{L^\infty(B_{r_1})}\leq\displaystyle  \frac{C\eps_0}{(1-\tau)^{C\eps_0^{p-1}}}.
\end{equation}

\item  In Step $2$, applying the Duhamel formulation of the equation 
satisfied by the cut-off of $v$ and estimate $(\ref{wr1})$, we
improve  estimate $(\ref{vr1})$ and prove that
for all $\tau \in [0,1)$,
\begin{equation}\label{vr}
\|v(\tau)\|_{L^\infty(B_{r_2})}\leq \frac{C\eps_0}{(1-\tau)^{C\eps_0^{p-1}}},
\end{equation}
for some $r_2\in(0, r_1)$.
\item In the final step, from  the inequalities  $(\ref{wr1})$ and
  $(\ref{vr})$, we prove
  that for all $\tau \in [0,1)$,
\begin{equation}\label{vr3}
\|v(\tau)\|_{L^\infty(B_{r_3})}+  \|w(\tau)\|_{L^\infty(B_{r_3})}\leq M\eps_0,
\end{equation}
for some $r_3\in(0, r_2]$ and a constant $M>0$.
\end{itemize}

\bigskip

{\bf{Step 1:}} Our starting point is the hypothesis \eqref{est_v}
given in Proposition \ref{prop_single_blow_up}. Using the Duhamel formulation satisfied by the cut-off of the gradient of the solution  and applying Proposition \ref{new_prop} for some $\eta>0$, we establish the following Lemma:
\begin{lem} \label{lemm_w}
  For any $r>0$, there exists $C_r>0$ such that for any $\alpha>0$,
  there exists $\bar \eps>0$ such that for any $\eps_0\in(0,
  \bar \eps)$, 
  if we have for all $\tau \in [0,1)$,
\begin{eqnarray}\label{est_v_w}
\|v\|_{L^\infty(B_r)}\leq \frac{\eps_0}{(1-\tau)^{\frac{1}{p-1}}}, \; \mbox{and}\quad  \|w\|_{L^\infty(B_r)}\leq \frac{C\eps_0}{(1-\tau)^{\alpha}},
\end{eqnarray}
then, for all $\tau \in [0,1)$,
\begin{eqnarray*}
 \displaystyle \|w\|_{L^\infty(B_{\frac{r}{2}})}\leq\left \{
\begin{array}{lcl}\displaystyle
C_{r} \frac{\eps_0}{(1-\tau)^{\alpha-\frac{\gamma}{2}}}, \; \mbox{if }\; \alpha\geq \gamma,\\
\displaystyle
C_{r} \frac{\eps_0}{(1-\tau)^{C_r \eps_0^{p-1}}}, \; \mbox{if }\; \alpha< \gamma,
\end{array}
\right.
\end{eqnarray*}
where $\gamma$ is introduced in \eqref{gamma}.
\end{lem}
Before proving this lemma, we need the following integral result from Giga and Kohn \cite{giga3} (see Lemma $2.2$ page 851 in \cite{giga3}):
\begin{lem}\label{lem_integ}
For $0<\alpha <1$, $\theta>0$ and $0\leq \tau<1$, the integral $$I(\tau)=\displaystyle \int_0^\tau (\tau -s)^{-\alpha}(1-s)^{-\theta}ds
$$
satisfies
\begin{itemize}
\item[i)] $I(\tau)\leq ((1-\alpha)^{-1}+(\alpha+\theta-1)^{-1})(1-\tau)^{1-\alpha-\theta},$ if $\alpha +\theta>1$,
\item[ii)] $I(\tau)\leq (1-\alpha)^{-1}+|\log(1-\tau)|,$ if $\alpha +\theta=1$,
\item[iii)]$I(\tau)\leq (1-\alpha-\theta)^{-1},$ if $\alpha +\theta<1$.
\end{itemize}
\end{lem}
With this lemma, we are ready to give the proof of Lemma \ref{lemm_w}.
\begin{proof}[Proof of Lemma \ref{lemm_w}]
Consider $w_r = \phi_r w$, where the cut-off $\phi_r$ is introduced
right before \eqref{vr1}.
Since $w$ satisfies $(\ref{eq_w})$, it follows that
$w_r$
satisfies the following equation: for all $\xi \in \R$ and $\tau \in [0,1)$,
\begin{eqnarray*}
 \partial_{\tau}w_r&=&\Delta w_{r} +w\Delta \phi_r-2\nabla (w\nabla \phi_r)+p|v|^{p-1}w_r\\
 &&+\mu (T-\t)^\gamma \nabla(\phi_r|w|\displaystyle \int_{B_0}|v|^{q-1})-\mu (T-\t)^\gamma \nabla\phi_r|w|\displaystyle \int_{B_0}|v|^{q-1}.
\end{eqnarray*}
The Duhamel equation satisfied by
$w_r$
is the following: for all $\tau \in [0,1)$,
\begin{eqnarray}\label{duh_w}
w_r(\tau)&=&S(\tau)w_r(0)+\int_0^\tau S(\tau-s) w\Delta \phi_r-2\int_0^\tau S(\tau-s)\nabla (w\nabla \phi_r)\nonumber\\
&&+p\int_0^\tau S(\tau -s)|v|^{p-1}w_r\nonumber\\
 &&+\mu (T-\t)^\gamma\int_0^\tau S(\tau -s) \nabla(\phi_r|w|\displaystyle \int_{B_0}|v|^{q-1})\nonumber\\
 &&-\mu (T-\t)^\gamma \int_0^\tau S(\tau -s) \nabla\phi_r|w|\displaystyle \int_{B_0}|v|^{q-1},
\end{eqnarray}
where  $S(t)$ is  the heat semigroup, and  $B_0$ already introduced
right after \eqref{gamma} is the ball of center $-\displaystyle 	K_0\sqrt{|\log(T-\t)|}\frac{x_0}{|x_0|}$ and radius $\displaystyle |\xi+K_0\sqrt{|\log(T-\t)|}\frac{x_0}{|x_0|}|.$
 We recall the following well-known smoothing effect of the heat semigroup:
\begin{equation}\label{Semigroup}
\|S(t)f\|_{L^\infty}\leq \|f\|_{L^\infty}, \quad \|\nabla S(t)f\|_{L^\infty}\leq \frac{C}{\sqrt{t}}\|f\|_{L^\infty},\quad \forall t>0, \; \forall f\in L^{ \infty}(\R^N).
\end{equation}
Taking the $L^\infty-$norm on the Duhamel equation \eqref{duh_w},
using the smoothness of
$\phi_r$
and inequalities $(\ref{Semigroup})$,   we get for all $\tau \in [0, 1)$:
\begin{eqnarray}\label{duha_w}
\|w_r(\tau)\|_{L^\infty}&\leq &\|w(0)\|_{L^\infty(B_r)}+C\int_0^\tau \|w\|_{L^\infty(B_r)}+C\int_0^\tau (\tau-s)^{-\frac{1}{2}}\|w\|_{L^\infty(B_r)}\nonumber\\
&&+p\int_0^\tau \|v\|_{L^\infty(B_r)}^{p-1}\|w_r\|_{L^\infty}\nonumber\\
 &&+C\mu (T-\t)^\gamma\int_0^\tau (\tau -s)^{-\frac{1}{2}} \|w\|_{L^\infty(B_r)}\displaystyle \|\int_{B_0}|v|^{q-1}\|_{L^\infty(B_r)}\nonumber\\
 &&+C\mu (T-\t)^\gamma \int_0^\tau \|w\|_{L^\infty(B_r)}\displaystyle \|\int_{B_0}|v|^{q-1}\|_{L^\infty(B_r)}.
\end{eqnarray}
Note that the various $C$ constants may depend on $r$.
Since $\gamma>0$ from \eqref{hyp_ga},
 we may fix some
\begin{equation}\label{defeta}
 \eta \in (0,\frac \gamma 2)
\mbox{ with in addition } \eta<\frac{\gamma-\alpha}2 \mbox{ if }\alpha<\gamma.
\end{equation} 
From assumption $(\ref{est_v_w})$   and inequality \eqref{2.6}, we
obtain for all $\tau \in [0,1)$,
\begin{eqnarray}
\|w_r(\tau)\|_{L^\infty}&\leq &C\eps_0+C\eps_0\int_0^\tau (\tau-s)^{-\frac{1}{2}} (1-s)^{-\alpha}ds\nonumber\\
&&+C\eps_0^{p-1}\int_0^\tau (1-s)^{-1}\|w_r(s)\|_{L^\infty}ds\label{ineq}\\
 &&+ C_\eta \eps_0(T-\t)^{\gamma-\eta}\int_0^\tau (\tau -s)^{-\frac{1}{2}} (1-s)^{-\frac{1}{2}+\gamma -\eta-\alpha}ds.\nonumber
\end{eqnarray}
In order to  apply Lemma \ref{lem_integ}, we distinguish $3$ cases:\\
{\bf{Case 1:}} When $\alpha >\frac{1}{2}$. Clearly, from \eqref{gamma}
and \eqref{hyp}, we have
\begin{equation}\label{g12}
0<  \gamma<\frac{1}{2}
\end{equation}
and
\begin{equation}\label{alphaeta}
  \alpha+\eta -\gamma>0.
\end{equation}
 Then, applying  item $i)$ of Lemma \ref{lem_integ}, we obtain
\begin{eqnarray*}
\|w_r(\tau)\|_{L^\infty}&\leq &C\eps_0(1-\tau)^{-\alpha+\frac{1}{2}} +C_\eta \eps_0 (T-\t)^{\gamma-\eta}(1-\tau)^{\gamma -\eta-\alpha}\\
&&+C\eps_0^{p-1}\int_0^\tau (1-s)^{-1}\|w_r(s)\|_{L^\infty}ds.
\end{eqnarray*}
Since  $\gamma-\alpha  -\eta <-\alpha +\frac{1}{2}$ from
\eqref{defeta} and
\eqref{g12},
using \eqref{defeta} again, we see that
\begin{eqnarray}\label{E}
\|w_r(\tau)\|_{L^\infty}&\leq &C_\eta\eps_0(1-\tau)^{\gamma -\eta-\alpha}+C\eps_0^{p-1}\int_0^\tau (1-s)^{-1}\|w_r(s)\|_{L^\infty}ds\\
&\leq&C_\eta\eps_0+C\eps_0^{p-1}\int_0^\tau
       (1-s)^{-1}\|w_r(s)\|_{L^\infty}ds\nonumber\\
  &&+ C_\eta \eps_0(\alpha +\eta-\gamma )\int_0^\tau (1-s)^{\gamma -\alpha -\eta -1}ds.\nonumber
\end{eqnarray}
Let us now recall the following Gronwall Lemma from Giga and Kohn \cite{giga3} (see Lemma 2.3 page 852 there):
\begin{lem}[A Gronwall lemma; Giga and Kohn \cite{giga3}] \label{lgron}
If $y(t), \,r(t)$ and $ q(t) $  are continuous functions defined on $[t_0, t_1]$ such that
\begin{eqnarray*}
y(t)\leq y_0+\int_{t_0}^t y(s) r(s) ds+ \int_{t_0}^t q(s) ds, \quad t_0\le t\le t_1,
\end{eqnarray*}
then
\begin{eqnarray*}
y(t)\leq \exp\left\{\int_{t_0}^t r(\tau) d\tau \right\}\Big[ y_0+ \int_{t_0}^t q(\tau) \exp\left\{-\int_{t_0}^\tau r(\sigma) d\sigma\right\}  ds\Big].
\end{eqnarray*}
\end{lem}
~\\
 Applying    Lemma \ref{lgron} to estimate \eqref{E}, we obtain for all $\tau \in [0,1)$,
\begin{eqnarray*}
\|w_r(\tau)\|_{L^\infty}\leq C_\eta\eps_0(1-\tau)^{-C \eps_0^{p-1}}\Big[1+ (\alpha +\eta-\gamma) \int_0^\tau (1-s)^{\gamma -\alpha -\eta +C\eps_0^{p-1}-1}ds\Big].
\end{eqnarray*}
From \eqref{alphaeta}, we see that we can choose $\eps_0$ small enough such that $\gamma -\alpha -\eta +C\eps_0^{p-1}<0 $, and  obtain for all $\tau \in [0,1)$,
\begin{eqnarray*}
\|w_r(\tau)\|_{L^\infty}\leq C_\eta \eps_0(1-\tau)^{\gamma -\alpha -\eta} .
\end{eqnarray*} 
Since  $\eta<\frac{\gamma}{2}$ from \eqref{defeta}, we see that for all $\tau \in[0,1)$,\begin{eqnarray*}
\|w_r(\tau)\|_{L^\infty}\leq C_\eta\eps_0(1-\tau)^{-\alpha+\frac{\gamma}{2}}.
\end{eqnarray*}
{\bf{Case 2:}} When $\gamma\le \alpha \leq \frac{1}{2}$. We see from
\eqref{defeta} that estimate \eqref{alphaeta}
still holds.
Therefore, using \eqref{ineq} and
applying items $i)$ and $iii)$  of   Lemma \ref{lem_integ}, we obtain for all $\tau \in [0,1)$,
\begin{eqnarray*}
  \|w_r(\tau)\|_{L^\infty}&\leq &
C\eps_0[1+|\log(1-\tau)|]+ C_\eta \eps_0(T-\t)^{\gamma-\eta}(1-\tau)^{\gamma -\eta-\alpha}\\
                          &&+C\eps_0^{p-1}\int_0^\tau (1-s)^{-1}\|w_r(s)\|_{L^\infty}ds\\
  &\le& C_\eta \eps_0(1-\tau)^{\gamma -\eta-\alpha}+C\eps_0^{p-1}\int_0^\tau (1-s)^{-1}\|w_r(s)\|_{L^\infty}ds,
\end{eqnarray*}
which is exactly the same inequality \eqref{E} encountered in Case
1. Thus, by the same argument, 
we obtain  for   $\eps_0$ small enough  and  for all $\tau \in [0,1)$,
\begin{eqnarray*}
\|w_r(\tau)\|_{L^\infty}\leq C_\eta\eps_0(1-\tau)^{-\alpha+\frac{\gamma}{2}}.
\end{eqnarray*}
{\bf{Case 3:}} When $\alpha <\gamma<\frac{1}{2}$.
From \eqref{defeta}, we see that
\begin{equation}\label{etaalpha}
\gamma -\alpha -\eta>0.
\end{equation}
Thanks to \eqref{ineq} together with item $iii)$ of Lemma \ref{lem_integ}, we obtain for all $\tau \in [0,1)$,
\begin{eqnarray*}
\|w_r(\tau)\|_{L^\infty}&\leq & C_\eta\eps_0
+C\eps_0^{p-1}\int_0^\tau (1-s)^{-1}\|w_r(s)\|_{L^\infty}ds.
\end{eqnarray*}
Using again Lemma \ref{lgron}, we see that for all $\tau \in [0,1)$,
\begin{eqnarray*}
\|w_r(\tau)\|_{L^\infty}\leq C_\eta\eps_0(1-\tau)^{-C \eps_0^{p-1}}.
\end{eqnarray*}
Since $\gamma <\frac 12$ from
\eqref{g12},
it is easy to see that
the cases 1, 2 and 3 mentioned above cover all possibilities. Thus, this concludes   the proof of Lemma \ref{lemm_w}.
\end{proof}
Starting from the hypothesis \eqref{est_v} stated in Lemma
\ref{prop_single_blow_up} and using a finite iteration, Lemma \ref{lemm_w} provides us with a 
positive constant  $r_1$ such that estimates $(\ref{vr1})$ and
$(\ref{wr1})$ are satisfied.

\bigskip

{\bf{Step 2}}: Here, we start from estimates \eqref{vr1} and
\eqref{wr1} we have just proved in Step 1. Then,
we give a Duhamel formulation satisfied by the cut-off
of the solution $v$. Using various estimates of Step $1$
together with Proposition \ref{new_prop}, we prove the following result:
\begin{lem}\label{lem_V2}
There exists a positive constant  $r_2<r_1$ such that
\begin{eqnarray}\label{vr2}
 \|v\|_{L^\infty(B_{r_2})}\leq C\eps_0(1-\tau)^{-C \eps_0^{p-1}}.
\end{eqnarray}
\end{lem}
\begin{proof}

  Let $v_{r_1}=\phi_{r_1} v$, where
  the cut-off
$\phi_{r_1}$ is defined right before \eqref{vr1}.
Using equation \eqref{eq_v},
we see that
$v_{r_1}$ satisfies   the following equation for all $\xi \in \R$ and $\tau \in [0,1)$:
\begin{eqnarray*}
 \partial_{\tau}v_{r_1}&=&\Delta v_{r_1} +v\Delta \phi_{r_1}-2\nabla (v\nabla \phi_{r_1})+|v|^{p-1}v_{r_1}
 +\mu (T-\t)^\gamma \phi_{r_1}|\nabla v|\displaystyle \int_{B_0}|v|^{q-1}.
\end{eqnarray*}
Using a Duhamel formulation, we see that 
for all $\tau \in [0,1)$,
\begin{eqnarray*}
v_{r_1}(\tau)&=&S(\tau)v_{r_1}(0)+\int_0^\tau S(\tau-s) v\Delta \phi_{r_1}-2\int_0^\tau S(\tau-s)\nabla (v\nabla \phi_{r_1})\\
&&+\int_0^\tau S(\tau -s)|v|^{p-1}v_{r_1}+\mu (T-\t)^\gamma\int_0^\tau S(\tau -s) \phi_{r_1}|\nabla v|\displaystyle \int_{B_0}|v|^{q-1}.
\end{eqnarray*}
Taking the $L^\infty-$norm,
using inequalities $(\ref{Semigroup})$  and  the smoothness of $\phi_{r_1}$,   we get for all $\tau \in [0, 1)$:
\begin{eqnarray}\label{duha_v}
&&\|v_{r_1}(\tau)\|_{L^\infty}\!\leq\! \|v(0)\|_{L^\infty(B_{r_1})}\!+C\!\int_0^\tau\!\! \|v\|_{L^\infty(B_{r_1})}\!+\!C\int_0^\tau\!\!\! (\tau-s)^{-\frac{1}{2}}\|v\|_{L^\infty(B_{r_1})}\\
&&+\int_0^\tau\!\! \|v\|_{L^\infty(B_{r_1})}^{p-1}\|v_{r_1}\|_{L^\infty}+C\mu (T-\t)^\gamma\int_0^\tau \!\! \|\nabla v\|_{L^\infty(B_{r_1})}\displaystyle \|\int_{B_0}|v|^{q-1}\|_{L^\infty(B_{r_1})}.\nonumber
\end{eqnarray}
From estimates \eqref{2.6},
\eqref{vr1} and \eqref{wr1},
we obtain  for all $\tau \in [0, 1)$,
\begin{eqnarray*}
\|v_{r_1}(\tau)\|_{L^\infty}&\leq &C\eps_0+C\eps_0\int_0^\tau
                                    (\tau-s)^{-\frac{1}{2}}
                                    (1-s)^{-\frac{1}{p-1}}ds\\
                            &&+C\eps_0^{p-1}\int_0^\tau (1-s)^{-1}\|v_{r_1}(s)\|_{L^\infty}ds\\
                            &&+ C \eps_0(T-\t)^{\gamma-\eta}\int_0^\tau  (1-s)^{-\frac{1}{2}+\gamma -\eta-C \eps_0^{p-1}}ds.
\end{eqnarray*}
Using \eqref{defeta} and \eqref{g12}, we write $\frac{1}{2}-\gamma
+\eta +C\eps_0^{p-1} \le
\frac 12-\frac \gamma 2 +C \eps_0^{p-1} \le
\frac 12  +C\eps_0^{p-1}<1$, for $\epsilon_0$ small enough. Recalling
that $p>3$ from \eqref{hyp}, 
we obtain by item $iii)$ of Lemma \ref{lem_integ}  for all $\tau \in [0, 1)$,
\begin{eqnarray*}
\|v_{r_1}(\tau)\|_{L^\infty}&\leq &C\eps_0+C\eps_0^{p-1}\int_0^\tau (1-s)^{-1}\|v_{r_1}(s)\|_{L^\infty}ds.
\end{eqnarray*}
Thanks to Lemma \ref{lgron}, we  conclude the proof of Lemma
\ref{lem_V2}, taking $r_2=\frac{r_1}2$.
\end{proof}
{\bf{Step 3}}: Estimates \eqref{vr} and \eqref{wr1} from Steps 1 and 2
are the starting point in this part. More precisely, we consider
$v_{r_2}=\phi_{r_2} v$  and
$w_{r_2}=\phi_{r_2} w$. Since
$r_2\le r_1$,
we write from
$(\ref{vr})$  and  $(\ref{wr1})$  for all $\tau \in [0, 1)$:
\begin{equation*}
\|v(\tau)\|_{L^\infty(B_{r_2})}\leq \frac{C\eps_0}{(1-\tau)^{C\eps_0^{p-1}}},
\end{equation*}
\begin{equation*}
\|w(\tau)\|_{L^\infty(B_{r_2})}\leq\displaystyle  \frac{C\eps_0}{(1-\tau)^{C\eps_0^{p-1}}}.
\end{equation*}
Plugging these estimates in  inequalities  $(\ref{duha_w})$ and
$(\ref{duha_v})$ (which both hold with $r$ and $r_1$ replaced by
$r_2$), taking $\epsilon_0$ small enough, we obtain  for all $\tau \in [0, 1)$,
\begin{eqnarray*}
\|v_{r_2}(\tau)\|_{L^\infty}&\leq &C\eps_0+C\eps_0^{p-1}\int_0^\tau (1-s)^{-C \eps_0^{p-1}(p-1)}\|v_{r_2}(s)\|_{L^\infty}ds,\\
\|w_{r_2}(\tau)\|_{L^\infty}&\leq &C\eps_0+C\eps_0^{p-1}\int_0^\tau (1-s)^{-C \eps_0^{p-1}(p-1)}\|w_{r_2}(s)\|_{L^\infty}ds.
\end{eqnarray*}
Taking   $\eps_0$       even smaller so that
$C_\eta \eps_0^{p-1}(p-1)<1$ and
using the Gronwall argument of 
Lemma \ref{lgron}, we deduce that  for all $\tau \in [0, 1)$,
 \begin{eqnarray*}
\|v_{r_2}(\tau)\|_{L^\infty}&\leq &M\eps_0,\\
\|w_{r_2}(\tau)\|_{L^\infty}&\leq &M\eps_0,
 \end{eqnarray*}
 for some $M>0$.
This concludes the proof of Proposition \ref{prop_single_blow_up} and part $1$ of Theorem \ref{main_thm} too.
\end{proof}
\section{Existence of the final profile}
In this section, we give the proof of part $2)$ of Theorem
\ref{main_thm}.

\bigskip

\begin{proof}[Proof of part 2) of Theorem \ref{main_thm}]
We consider $u$ a solution of equation $(\ref{eq_u})$ which blows up at the origin and only there in finite time $T>0$.
Adapting the method used by Merle \cite{merle} and Zaag \cite{zaag}, we prove the existence of a blow-up
final profile $u^\ast$ such that $u(x,t)\to u^\ast $ in $C^1$  of
every compact  of $\R^N\setminus \{0\}$.

\medskip

Consider $x_0\neq 0$ to be taken small enough later. The proof is performed in the
rescaled variable framework, $v(x_0, \xi, \tau)$ and $w(x_0, \xi, \tau)$ introduced in \eqref{v}, \eqref{w} and
\eqref{x_0}. Let us first put together the estimates we
already have:

\bigskip

\noindent\textbf{1-Initialization for $v$ and $w$ at $\tau=0$}.
From estimates $(\ref{b_u})$, $(\ref{b_nabla_u})$, the definitions
\eqref{v} and \eqref{w} of $v$ and $w$, we have the following:
\begin{equation}\label{b_v}
\displaystyle \sup_{|\xi|\leq 6 |\log(T-\t)|^{\frac{1}{4}}}|v(x_0, \xi, 0)-f(K_0)|\leq \frac{C}{|\log(T-\t)|^{\frac{1}{2}}},
\end{equation}
\begin{equation}\label{b_w}
\displaystyle \sup_{|\xi|\leq 6 |\log(T-\t)|^{\frac{1}{4}}}|w(x_0, \xi, 0)-\frac{1}{\sqrt{|\log(T-\t)|}} \nabla f(K_0)|\leq \frac{C}{|\log(T-\t)|^{\frac{1}{2}}},
\end{equation}
where $t_0(x_0)$ is defined in \eqref{x_0}.

\bigskip

\noindent\textbf{2-A rough bound on $v$ and $w$ for $\tau \in [0,1)$}.
We claim that for all $\tau \in [0,1)$ and $|\xi|\leq 6 |\log(T-\t)|^{\frac{1}{4}}$, we have
\begin{eqnarray}
\displaystyle (1\!-\!\tau)^{\frac{1}{p-1}}| v(x_0, \xi, \tau)|+
  \sqrt{1\!-\!\tau }\; | w(x_0, \xi, \tau)|&\leq&
                                                  f\left(\frac{K_0}{2}\right)+\frac{C}{\sqrt{|\log(T-\t)|}} |\nabla f(K_0)|\nonumber\\
                                                                                                         &+&\frac{C}{|\log(T-\t)|^{\frac{1}{2}}}\equiv\eps_1(K_0, x_0).\label{M0}
\end{eqnarray}
The proof is exactly the same as for the
calculation displayed in the proof of part 1 of Theorem
\ref{main_thm} on page \pageref{ProofTh1}, though we consider here
$\xi$ is some larger ball.

\bigskip

\noindent\textbf{3-A uniform bound on $v$ and $w$ for $\tau \in [0,1)$}.
Let us choose $K_0$ large enough and $|x_0| $ small enough such that
$\eps_1(K_0, x_0)<\bar \eps$, where $\bar \eps$
is the constant introduced in Proposition \ref{prop_single_blow_up}.\\
Applying that proposition,
we obtain
\begin{equation}\label{M1}
\sup_{|\xi|\leq 5|\log(T-\t)|^{\frac{1}{4}} }|v(x_0, \xi, \tau)|+|w(x_0, \xi, \tau)|\leq M\eps_1(K_0, x_0)\equiv M_1.
\end{equation}
Now, using the above
mentioned information,
we proceed in three steps:
\begin{itemize}
\item First, we prove that
  for all $\tau\in [0,1)$, $|\xi|
  \le 2|\log(T-\t)|^\frac{1}{4}$,
\begin{equation}\label{S1}
\displaystyle \|w(x_0,\xi, \tau)\|_{L^\infty}\leq
\frac{C}{|\log(T-\t)|^{\frac 14}}.
\end{equation}
\item Next, we prove that
  for all $\tau\in [0,1)$, $ |\xi|\le |\log(T-\t)|^\frac{1}{4}$,
  \begin{equation}\label{S2}
    \displaystyle \|v(x_0,\xi, \tau)-v_{K_0}(\tau)\|_{L^\infty}\leq
    \frac{C}{|\log(T-\t)|^{\frac 14}},
  \end{equation}
where $v_{K_0}(\tau)=((p-1)(1-\tau)+bK_0^2)^{-\frac{1}{p-1}}$ is the solution of the ordinary differential equation $v'_{K_0}(\tau)=v_{K_0}^p(\tau)$, with initial data $v_{K_0}(0)=f(K_0)$.
\item Finally, we use 
classical parabolic regularity and the definition of $\t$, to get the equivalent of $u^\ast$ as $x_0\to 0$.
\end{itemize}

{\bf{Step 1 : Smallness of $w$}}

 Arguing as in Proposition \ref{prop_single_blow_up}, we introduce
 $\phi$ a $C^\infty$ cut-off function, with $supp(\phi)\subset
 B(0,1)$, $0\leq \phi\leq 1$ and $\phi=1$ on $B(0,
 \frac{1}{2})$. We also introduce for every
 $r>0$,   $$\phi_r(\xi)=\displaystyle
 \phi(\frac{\xi}{r|\log(T-\t)|^{\frac{1}{4}}}).$$
Note that
\begin{equation}\label{nabla_phi}
\|\nabla \phi_r\|_{L^\infty}\leq \frac{C}{r|\log(T-\t)|^{\frac{1}{4}}},
\end{equation}
and \begin{equation}\label{delta_phi}
\|\Delta \phi_r\|_{L^\infty}\leq \frac{C}{r^2|\log(T-\t)|^{\frac{1}{2}}}.
\end{equation}
Introducing  for $r=4$, $w_4=\phi_4 w$,  and arguing as for
$(\ref{duha_w})$, taking the $L^\infty$-norm on the Duhamel equation satisfied by $w_2$,  we obtain  for all $\tau \in [0, 1)$,
\begin{eqnarray}
\|w_4(\tau)\|_{L^\infty}\!\!\!&\leq & \!\!\!\|w_4(0)\|_{L^\infty}\!+\!C\!\int_0^\tau\!\! \|\Delta \phi_4\|_{L^\infty} \|w\|_{L^\infty(B)}\!+\!C\!\int_0^\tau\!\! (\tau-s)^{-\frac{1}{2}}\|\nabla \phi_4\|_{L^\infty}\|w\|_{L^\infty(B)}\nonumber\\
&&+p\int_0^\tau \|v\|_{L^\infty(B)}^{p-1}\|w_{4}\|_{L^\infty}\nonumber\\
 &&+C\mu (T-\t)^\gamma\int_0^\tau (\tau -s)^{-\frac{1}{2}} \|w\|_{L^\infty(B)}\displaystyle \|\int_{B_0}|v|^{q-1}\|_{L^\infty(B)}\nonumber\\
 &&+C\mu (T-\t)^\gamma \int_0^\tau \|\nabla \phi_4\|_{L^\infty}\|w\|_{L^\infty(B)}\displaystyle \|\int_{B_0}|v|^{q-1}\|_{L^\infty(B)},
\end{eqnarray}
where
 $B_0$ already introduced
right after \eqref{gamma} is the ball of center $-\displaystyle 	K_0\sqrt{|\log(T-\t)|}\frac{x_0}{|x_0|}$ and radius $\displaystyle |\xi+K_0\sqrt{|\log(T-\t)|}\frac{x_0}{|x_0|}|, $  and $ B$ is the ball of center $0$ and radius 
        $4|\log(T-\t)|^{\frac{1}{4}} .$\\
Consider $\eta\in (0, \gamma)$.
Using \eqref{b_w},  $(\ref{M1})$, $(\ref{nabla_phi})$,
$(\ref{delta_phi})$  and \eqref{2.6} we deduce
for all $\tau \in [0, 1)$:
\begin{eqnarray*}
&&\|w_{4}(\tau)\|_{L^\infty}\leq \frac{C}{|\log(T-t_0)|^{\frac{1}{2}}}+\frac{CM_1}{|\log(T-\t)|^{\frac{1}{2}}}+ \frac{CM_1}{|\log(T-\t)|^{\frac{1}{4}}}\nonumber\\
 &&+C_\eta(T-\t)^{\gamma-\eta} M_1\Big(\int_0^\tau \big( (\tau -s)^{-\frac{1}{2}} +\frac{C}{|\log(T-\t)|^{\frac{1}{4}}}\big) (1-s)^{\gamma-\frac{1}{2}-\eta}ds\Big)\nonumber\\&&+ p M_1^{p-1}\int_0^\tau \|w_{4}\|_{L^\infty}.
\end{eqnarray*}
If we remark  that $M_1\le 1$  (see\eqref{M1} and \eqref{M0}) and
\begin{equation}\label{H_t}
\displaystyle (T-\t)^{\gamma-\eta}\leq \frac{C}{|\log(T-\t)|^{\frac 14}},
\end{equation}
 for
$x_0$ small enough and $K_0$ large enough,   
then by
Lemma \ref{lem_integ},  we obtain    for all $\tau \in [0, 1)$,
\begin{eqnarray*}
\|w_{4}(\tau)\|_{L^\infty}\leq \frac{C}{|\log(T-\t)|^{\frac 14}}+ p\int_0^\tau \|w_{4}(s)\|_{L^\infty}ds.
\end{eqnarray*}
Applying   Lemma \ref{lgron}, we deduce that  for all $\tau \in [0, 1)$,
\begin{eqnarray*}
\|w_{4}(\tau)\|_{L^\infty}\leq \frac{C}{|\log(T-\t)|^{\frac 14}}.
\end{eqnarray*}
Thus, $(\ref{S1})$ follows.

\bigskip

{\bf{Step 2 : Sharp behavior of $v$}}

Let
\begin{equation}\label{defvk0}
 v_{K_0}(\tau)=((p-1)(1-\tau)+bK_0^2)^{-\frac{1}{p-1}}
\end{equation}
  be the solution of the ordinary differential equation $v'_{K_0}(\tau)=v_{K_0}^p(\tau)$, with initial data $v_{K_0}(0)=f(K_0)$.\\
Introducing $\psi = v-v_{K_0}$ and
$\psi_2=\phi_2\psi$, we see from equation \eqref{eq_v}  that
$\psi_2$ satisfies the following equation   for all $\xi \in \R$ and for all $\tau \in [0, 1)$:
\begin{eqnarray*}
 \partial_{\tau}\psi_{2}=\Delta \psi_{2} +\psi\Delta \phi_{2}-2\nabla (\psi\nabla \phi_{2})+a \psi_2+\mu (T-\t)^\gamma \phi_{2}|\nabla v|\displaystyle \int_{B_0}|v|^{q-1},
\end{eqnarray*}
where \begin{equation*}
a(x_0, \xi , \tau)=\left\{ \begin{array}{cc}
\displaystyle \frac{|v|^{p-1}v-v_{K_0}^p}{v-v_{K_0}},\quad \mbox{if}\; v\neq v_{K_0},\\
pv_{K_0}^{p-1}\quad \mbox{otherwise}.
\end{array}\right.
\end{equation*}
Using \eqref{M1}, \eqref{M0} and \eqref{defvk0}, we see that
\[
\sup_{|\xi|\le 2|\log(T-t_0(x_0))|^{\frac 14}, \tau\in[0,1)}|a(x_0, \xi , \tau)|\equiv\eta_2\to 0, 
\mbox{ as }|x_0|\to 0\mbox{ and }K_0\to +\infty.
\]
Taking the $L^\infty$-norm on the Duhamel equation satisfied by
$\psi_2$, and using  estimates \eqref{b_v}, \eqref{2.6}, $(\ref{M1})$, $(\ref{nabla_phi})$ and  $(\ref{delta_phi})$, we get  for all $\tau \in [0, 1)$,
\begin{eqnarray*}
\|\psi_{2}(\tau)\|_{L^\infty}&\leq &\frac{C}{|\log(T-t_0)|^{\frac{1}{2}}}+\frac{CM_1}{|\log(T-\t)|^{\frac{1}{2}}}+ \frac{CM_1}{|\log(T-\t)|^{\frac{1}{4}}}\nonumber\\
 &&+C_\eta(T-\t)^{\gamma-\eta} M_1\int_0^\tau (1-s)^{\gamma-\frac{1}{2}-\eta}ds+ C\eta_2 \int_0^\tau \|\psi_{2}(s)\|_{L^\infty}ds
\end{eqnarray*}
Using \eqref{H_t},
we obtain   for all $\tau \in [0, 1)$,
\begin{eqnarray*}
\|\psi_{1}(\tau)\|_{L^\infty}&\leq &\frac{C}{|\log(T-t_0)|^{\frac 14}}+ C\eta_2 \int_0^\tau \|\psi_{1}\|_{L^\infty}
\end{eqnarray*}
By Lemma \ref{lgron}, we deduce that  for all $\tau \in [0, 1)$,
\begin{eqnarray*}
\|\psi_{1}(\tau)\|_{L^\infty}\leq \frac{C}{|\log(T-\t)|^{\frac 14}}.
\end{eqnarray*}
This concludes   the proof of $(\ref{S2})$.

\bigskip

{\bf Step 3 : Conclusion of the proof of Part 2 of Theorem
    \ref{main_thm}}

  Using  $(\ref{S1})$, $(\ref{S2})$ and classical parabolic regularity, we see that
$$\forall \tau \in [\frac{1}{2}, 1), \;\; |\xi|\leq
\frac{1}{2}|\log(T-\t)|^{\frac{1}{4}}, \;\;|\partial_\tau v(x_0, \xi ,
\tau)|+|\partial_\tau w(x_0, \xi , \tau)|\leq C.$$
Therefore, 
we have the existence of limits for $v(x_0, 0, \tau)$ and $w(x_0, 0,
\tau)$ as $\tau\to 1$ for $|x_0|$ sufficiently small, on the one hand.\\
On the other hand, since neither $u$ nor $\nabla u$ blow up outside the
origin, using again classical parabolic regularity, we derive the
existence of a limiting profile $u^*$ such that $u(x,t) \to u^*(x)$ as
$t\to T$, in $C^1$ of every compact of $\R^N\backslash \{0\}$ (see
Merle \cite{merle} for a similar argument).\\
Letting $\tau \to 1$ in $(\ref{S1})$ and $(\ref{S2})$, and using the definitions $(\ref{v})$ and $(\ref{scaling})$,  we have
$$\displaystyle u^\ast(x_0)=\lim_{t\to T}u(x_0, t)= \lim_{\tau \to 1}\frac{v(x_0,0, \tau)}{(T-\t)^{\frac{1}{p-1}}}\sim (bK_0^2)^{-\frac{1}{p-1}}(T-\t)^{-\frac{1}{p-1}}, \; \mbox{as}\; x_0\to 0,$$
and
\begin{eqnarray*}
\displaystyle |\nabla u^\ast(x_0)|&=&|\lim_{t\to T} \nabla u(x_0, t)|=|\lim_{\tau \to 1}\frac{w(x_0,0, \tau)}{(T-\t)^{\frac{1}{p-1}+\frac{1}{2}}}|\\
&\leq& \frac{C}{|\log(T-\t)|^{\frac 14}(T-\t)^{\frac{1}{p-1}+\frac{1}{2}}}.
\end{eqnarray*}
From  the definition $(\ref{x_0})$ of $\t$, we have
$$\log|x_0|\sim\frac{1}{2}\log(T-\t),$$ and
$$T-\t\sim \displaystyle \frac{|x_0|^2}{2K_0^2|\log(x_0)|} ,\;\mbox{as}\; x_0\to 0.$$
Hence, $$u^\ast (x_0)\sim \left(\frac{b|x_0|^2}{2|\log(x_0)|}\right)^{-\frac{1}{p-1}} ,\;\mbox{as}\; x_0\to 0,$$
and $$\displaystyle |\nabla u^\ast(x_0)|\leq 
  |x_0|^{-\frac{p+1}{p-1}}|\log|x_0||^{\frac {p+3}{4(p-1)}}.$$
  This concludes   the proof of part 2) of Theorem \ref{main_thm}.
\end{proof}

\noindent\textbf{Address}:\\
- Department of Mathematics, Faculty of Sciences
of Sfax, BP 1171, Sfax 3000, Tunisia. \\
\vspace{-7mm}
\begin{verbatim}
 e-mail: bouthaina.abdelhedi@fss.usf.tn
\end{verbatim}
- Universit\'e Sorbonne Paris Nord, Institut Galil\'ee, 
Laboratoire Analyse, G\'eom\'etrie et Applications, CNRS UMR 7539,
99 avenue J.B. Cl\'ement, 93430 Villetaneuse, France.\\
\vspace{-7mm}
\begin{verbatim}
e-mail: Hatem.Zaag@univ-paris13.fr
\end{verbatim}
\end{document}